\documentclass[11pt]{amsart}

\usepackage{latexsym}
\usepackage{amsmath}
\usepackage{amssymb}
\usepackage{amsthm}
\usepackage[colorlinks=true,linkcolor=blue,citecolor=blue]{ hyperref} % DELETE ,hypertex=true WHEN PRODUCING A pdf FILE

\newtheorem{Thm}{Theorem}[section]

\newtheorem{Lem}[Thm]{Lemma}
\newtheorem{Prop}[Thm]{Proposition}

\theoremstyle{definition}

\theoremstyle{remark}
\newtheorem{remark}{Remark}

\newcommand{\D}{\displaystyle}

\newcommand{\DD}{\mathbb D}
\newcommand{\C}{\mathbb C}

\newcommand{\Tt}{\mathbb T}

\def\ldots{\mathinner{\ldotp\ldotp\ldotp}}
\def\ldots{\mathinner{\cdotp\cdotp\cdotp}}

\DeclareMathOperator{\mon}{mon}

\numberwithin{equation}{section}

\begin{document}

\title{Hypercontractivity of the Bohnenblust-Hille inequality for polynomials and  multidimensional Bohr radii}
\author{Andreas Defant}
\address{Institut of mathematics, Carl von Ossietzky University, D--$26111$, Oldenburg,
Germany} \email{defant@mathematik.uni-oldenburg.de}

\author{Leonhard Frerick}
\address{Fachbereich IV - Mathematik, Universit\"{a}t Trier, D-54294 Trier}
\email{frerick@uni-trier.de}

\begin{abstract}
In 1931 Bohnenblust and Hille proved  that for each  m-homogeneous polynomial
$\sum_{|\alpha| = m} a_\alpha z^\alpha$ on $\C^n$ the $\ell^{\frac{2m}{m+1}}$-norm of its coefficients
is bounded from above by  a constant $C_m$ (depending only on the degree $m$) times the sup norm of the polynomial on the
polydisc $\DD^n$. We prove that this inequality is hypercontractive in the sense that the optimal constant $C_m$
is  $\leq C^m$ where $C \geq 1$ is an absolute constant.
>From this we derive that the Bohr radius $K_n$ of the $n$-dimensional polydisc in
$\mathbb{C}^n$ is up to an absolute constant $\geq \sqrt{\log n/n}$; this result was independently
and with a differnt proof discovered by Ortega-Cerd{\`a}, Ouna\"ies and Seip in \cite{OrOuSe}.
An alternative approach  even allows to
prove  that the Bohr radius $K_n^p$,  $1 \leq p \leq \infty $ of the unit ball of
$\ell_n^p \,,$ is asymptotically $ \geq  ( \log n/n
) ^{1-1/ \min (p,2)}$. This shows that the upper bounds for $K_n^p$ given by Boas and Khavinson from \cite{BoaKha} are optimal.
\end{abstract}

\maketitle \markboth{A.DEFANT AND L.FRERICK}{HYPERCONTRACTIVITY AND BOHR RADII}

\section{Introduction and main results}

In 1930 Littlewood proved the following (innocent looking) inequality
which is nowadays often cited as Littlewood's 4/3-inequality: For every bilinear form $A :
\C^n \times  \C^n \rightarrow \mathbb{C}$  we have
\begin{equation*} \label{little}
\bigg( \sum_{i,j} | A(e_{i}, e_{j}) |^{4/3} \bigg)^{3/4} \leq \sqrt{2} \sup_{x,y \in \DD^n}|A(x,y)|\,,
\end{equation*}
and  the exponent $4/3$ is optimal; here as usual  $\DD$ denotes the open unit disc in $\C$.
It seems  that Bohnenblust and Hille in 1931 immediately realized the importance of this results (and the techniques used in its
proof) for the study of lower bounds for the maximal width  $T$ of the strip of uniform but
non-absolute convergence of Dirichlet series $\sum a_n 1/n^s$. Bohr in 1913 in his article \cite{Bo} had shown that $T \leq 1/2$, and the
in the years following the question whether this estimate was optimal or not became well known under the name ``Bohr's  absolute convergence problem''.
Closing a long story  Bohnenblust-Hille   in their ingenious article \cite{Bohnenblust-Hille} proved  that
 in fact $T=1/2$.

The crucial  step in their solution is formed by an $m$-linear version of Littlewood's result together with its symmetrization for polynomials:
For each $m$ there is a constant $C_m \geq 1$
such that for each $n$ and for
each $m$-linear mapping $A : \C^n \times \cdots \times \C^n \rightarrow
\mathbb{C}$ we have
\begin{equation} \label{eq:Bohn-Hille}
\bigg( \sum_{i_1, \ldots,i_m } |A(e_{i_{1}}, \dots   ,
e_{i_{m}})|^{\frac{2m}{m+1}} \bigg)^{\frac{m+1}{2m}} \leq C_m \sup_{x_i \in \DD^n}|A(x_1, \dots, x_m)|\,,
\end{equation}
and  again the exponent $\frac{2m}{m+1}$
is optimal. Moreover, if $C_m$ stands for the best constant, then
the original  proof gives that  $C_m \leq m^{(m+1)/(2m)} 2^{(m-1)/2}$.
This inequality was forgotten
for long time and re-discovered by Davie \cite{Da73} and Kaijser \cite{Ka78}, see also \cite{Bl};
their proofs are (slightly) different from the original one and give the better constant
\begin{equation} \label{constant}
C_m \leq \sqrt{2}^{m-1}\,.
\end{equation}

In order to solve Bohr's ``absolute convergence problem'' Bohnenblust and Hille in fact needed a symmetric version of  \eqref{eq:Bohn-Hille}.
They used (or better invented) polarization and deduced from \eqref{eq:Bohn-Hille} that for each $m$ there is some constant $D_m\geq 1$
such that for each $n$ and for each m-homogeneous polynomial $\sum_{|\alpha| = m} a_\alpha z^\alpha$
on $\C^n$
\begin{equation} \label{poly}
\big( \sum_{|\alpha| = m} |a_\alpha|^{\frac{2m}{m+1}} \big)^{\frac{m+1}{2m}} \leq   D_m
 \sup_{z \in \DD^n} |\sum_{|\alpha| = m} a_\alpha z^\alpha|\,;
\end{equation}
and again they showed through a  highly non trivial argument that the
exponent $\frac{2m}{m+1}$ can not be improved. A nowadays standard argument
allows to deduce from \eqref{constant} and an estimate for the
polarization constant of $\ell_\infty$ due to
 Harris \cite{Ha75} that
\begin{equation*} \label{constantsym}
D_m \leq( \sqrt{2} )^{m-1}  \frac{m^{m/2}(m+1)^{\frac{m+1}{2}}}{2^{m} (m!)^{\frac{m+1}{2m}}}
\end{equation*}
(see e,g. \cite[Section 4]{DeSe}), and using  Sawa's Khinchine type inequality for Steinhaus variables
Queff{\'e}lec  in \cite[Theorem III-1]{Queffelec} gets the slightly better
estimate :
\[
D_m \leq \big( \frac{2}{\sqrt{\pi}} \big)^{m-1}  \frac{m^{m/2}(m+1)^{\frac{m+1}{2}}}{2^{m} (m!)^{\frac{m+1}{2m}}}.
\]
Our first main result is  the following substantial improvement. We show that the Bohnenblust-Hille inequality \eqref{poly}
for polynomials in fact is hypercontractive in the sense that its best constant $D_m$ for some absolute constant
$C \geq 1$ satisfies $D_m \leq C^m$.
\begin{Thm} \label{mainresult}
There is a $ C \geq 1$ such that for all $ m,n  $
\[
\big( \sum_{|\alpha| = m} |a_\alpha|^{\frac{2m}{m+1}} \big)^{\frac{m+1}{2m}} \leq C^m \sup_{z \in \DD^n} |\sum_{|\alpha| = m} a_\alpha z^\alpha|\,,
\]
where $\sum_{|\alpha| = m} a_\alpha z^\alpha$ is an m-homogeneous polynomial on $\C^n$.
\end{Thm}

Let us indicate that this result (see section 3 for the proof) has some far reaching consequences.
Given an $n$-dimensional Banach space  $X_n = (\C^n, \|\cdot\|)$ for which
the $e_k$'s form a $1$-unconditional basis, we use this result to  estimate
$n$-dimensional Bohr radii of the open unit ball $B_{X_n}$ in $X_n$, and to estimate
unconditional basis constant $\chi_{\mon}(\mathcal{P}(^mX_n))$
of the monomials $z^\alpha$ in the Banach space
$\mathcal{P}(^mX_n))$ of all $m$-homogeneous
polynomials.

Recall that the Bohr radius $K(B_{X_n})$ of the open unit ball $B_{X_n}$ (a Reinhardt domain)
is the infimum of all $r \geq 0$ such that for each
holomorphic function $f = \sum_\alpha a_\alpha z^\alpha$ on $B_{X_n}$ we have
\[
\sup_{z \in r B_{X_n}} \sum_\alpha |a_\alpha z^\alpha| \leq
\sup_{z \in B_{X_n}}|\sum_\alpha a_\alpha z^\alpha|\,.
\]
The unconditional basis constant $\chi_{\mon}(\mathcal{P}(^mX_n))$
of the monomials $z^\alpha$ in
$\mathcal{P}(^mX_n))$ by definition is the best constant $C \geq 1$ such that for every
$m$-homgeneous polynomial $\sum_{|\alpha| = m} a_\alpha z^\alpha$ on $\C^n$ and any choice of scalars $\varepsilon_\alpha$ with
$|\varepsilon_\alpha| \leq 1$ we have
\[
\sup_{z \in B_{X_n}} |\sum_{|\alpha| = m} \varepsilon_\alpha a_\alpha z^\alpha|
\leq C \sup_{z \in B_{X_n}} |\sum_{|\alpha| = m} a_\alpha z^\alpha|\,.
\]
Asymptotic estimates for unconditional basis constants of spaces of $m$-homogeneous polynomials  on $X_n = \ell_p^n$
were given
in \cite[Theorem 3]{DeDiGaMa}; as usual $\ell^n_p$, $1\le p\le \infty$ and
$n\in \mathbb{N}$, stands for $\C^n$ together with the $p$-norm $\|z\|_p :=
(\sum_k |z_k|^p)^{1/p}$ (with the obvious modification for $p =
\infty$). These results were improved in \cite[Lemma3.1]{DeFr} where it is shown that
$$\chi_{\mon}(\mathcal{P}(^m\ell^n_p)) \leq C^m n^{(m-1)( 1-\frac{1}{\min \{p,2\} })} \,,$$
$C\geq 1$ some absolute constant. Our second main result is:

\begin{Thm} \label{mainII}
There is a constant $C \geq 1$ such that for each $1 \leq p \leq \infty$
and all $m,n$
\[
\chi_{\mon}(\mathcal{P}(^m \ell^p_n)) \leq C^m
\big(1+\frac{n}{m} \big)^{(m-1)( 1-\frac{1}{\min \{p,2\} })}.
\]
\end{Thm}

During the preparation of this manuscript we were informed that for $p=\infty$ and $n > m^2 >1$
this result has been obtained independently and with a substantially different proof
by Ortega-Cerd{\`a}, Ouna\"ies and Seip in their very recent article
\cite[Theorem 1]{OrOuSe}. There it is
presented as an upper estimate of the Sidon constant for the index set of nonzero $m$-homogeneous
polynomials in $n$ complex variables (see also \eqref{Sidon} and \eqref{lower} below for equivalent formulations).
Several {\bf remarks on Theorem \ref{mainresult}} follow:

\vspace{2mm}

\noindent{\bf (1)} Let us first
indicate how for  $p=\infty$ the preceding theorem can be deduced as an immediate consequence
of the hypercontractivity of the constant in Theorem \ref{mainresult}: Clearly we have
$$  \chi_{\mon}(\mathcal{P}(^m \ell_\infty^n)) =
\sup \{  \sum_{|\alpha| = m} |a_\alpha| : \sup_{z \in \DD^n} |\sum_{|\alpha| = m} a_\alpha z^\alpha  | \leq 1 \}\,, $$
hence by H\"older's
inequality for each polynomial $\sum_{|\alpha| = m} a_\alpha z^\alpha$
$$
\sum_{|\alpha| = m} |a_\alpha| \leq \big(  \sum_{|\alpha|=m} 1 \big)^{\frac{m-1}{2m}} \big( \sum_{|\alpha| = m} |a_\alpha|^{\frac{2m}{m+1}} \big)^{\frac{m+1}{2m}}\,.
$$
But then  Theorem \ref{mainresult} and
a straight forward calculation using Stirling's formula (see also \eqref{dimension}) as desired show that there is a constant
 $ C \geq 1$ such that for all m-homogeneous polynomials $\sum_{|\alpha| = m} a_\alpha z^\alpha $ on $\C^n$
we have
\begin{equation} \label{Sidon}
\sum_{|\alpha| = m} |a_\alpha| \leq C^m \big(1+\frac{n}{m} \big)^{\frac{m-1}{2}}\sup_{z \in \DD^n}
|\sum_{|\alpha| = m} a_\alpha z^\alpha|\,.
\end{equation}

\vspace{2mm}

\noindent{\bf (2)} From \cite[Lemma 3.2]{DeFr} we know that there is some constant $C \geq 1$ such
for each Banach space  $X_n = (\C^n,\|\cdot\|)$  for which the  $e_k$'s form a
1-unconditional basis and each $m$,
\begin{align*}
   \chi_{\mon} (\mathcal{P}(^mX_n)) \;\le\; \chi_{\mon}(\mathcal{P}(^m\ell^n_\infty)) \,.
\end{align*}
Hence, once in Theorem \ref{mainII} the case  $p=\infty$  is proved, the case $2 \leq p$ follows.

\vspace{3mm}

\noindent{\bf (3)} Moreover, for $2 \leq p$  Theorem \ref{mainII}  is optimal in the following sense: Given $2 \leq p \leq \infty$, we have
\begin{equation} \label{lower}
\chi_{\mon}(\mathcal{P}(^m \ell^p_n))\,\,\sim \,\,
\left\{
\begin{array}{lll}
\frac{1}{\sqrt{m^{m-1}}}\sqrt{n^{m-1}} & \mbox{ if } n > m \\
\,\,1 & \mbox{ if } m \geq n\,,
\end{array}
\right.
\end{equation}
where $A_{mn}\sim B_{mn}$ means that there is some constant $C\geq 1$ such that for every $m,n$ we have
$1/C^m A_{mn} \leq B_{mn} \leq C^m A_{mn}$; indeed, this follows from an easy calculation since by a probabilistic estimate from \cite[(4.4)]{DeGarMa_BorhLoc} we know that for each such $p$ there is  some constant $d_p>0$ such that for every $m,n$
\[
 \dfrac{\sqrt{n^{m-1}}}{d_p\,\sqrt{m!}2^{\frac{3m}{2}-\frac{1}{2}} m^{\frac{3}{2}}}  \leq \chi_{\mon}(\mathcal{P}(^m \ell^p_n))\,.
\]

\vspace{2mm}

\noindent{\bf (4)} The case  $p \leq 2$ in Theorem \ref{mainII} needs a different approach of independent interest.
This approach improves ideas from \cite{DeDiGaMa},  will be given in section 6 based on the results from the sections 4 and 5,
and does still
cover the case $p \geq 2$. Invariants from local Banach space theory as  Gordon-Lewis and projection constants are involved.

\vspace{2mm}

Let us finally turn to multidimensional Bohr radii. In \cite[Theorem~2.2]{DeGarMa_BorhLoc} a basic link between
Bohr radii and unconditional basis constants is given: For every  $n$-dimensional Banach space  $X_n = (\C^n, \|\cdot\|)$ for which
the $e_k$'s form a $1$-unconditional basis we have
\begin{equation} \label{eq2.1}
  \frac{1}{3R(X_n)}\,
  \le K(B_{X_n})
  \le \min \Bigg(\frac{1}{3}, \; \frac{1}{R(X_n)}\Bigg)\,,
\end{equation}
where $R(X_n):= \sup_m \chi_{\mon}(\mathcal{P}(^mX_n))^{\frac{1}{m}}$\,.
This means that estimates for unconditional basis constants of $m$-homoge\-neous polynomials
always lead to estimates for multidimensional Bohr radii. For $n=1$ we obtain
Bohr's famous power series theorem
\begin{equation*}\label{Bohr}
K(\mathbb{D})= \frac{1}{3}\,
\end{equation*}
from  \cite{Bohr}, and hence  \eqref{eq2.1} can be seen as an abstract extension of Bohr's theorem
(let us remark that Bohr discovered his power series theorem  in the context of the above mentioned ``absolute convergence problem'').

By results of Aizenberg, Boas, Dineen, Khavinson,
Timoney and ourselves  from \cite{Aiz}, \cite{Boas}, \cite{BoaKha}, \cite{DeFr}, \cite{DinTim1}  there is a constant $C \geq 1$ such that for
all $1 \leq p \leq \infty$ and all $n$
\begin{equation}\label{Boas}
\frac{1}{C} \bigg( \frac{\log n}{n \log \log n} \bigg) ^{1-\frac{1}{\min (p,2)}}
\leq K(B_{\ell^n_{p}})
\leq C \bigg( \frac{\log n}{n} \bigg) ^{1-\frac{1}{\min
(p,2)}}\,.
\end{equation}
Our third main result is the following improvement:
\begin{Thm} \label{mainIII}
There is a constant $C > 0$ such that for each $1 \leq p \leq \infty$
and all $n$
$$ \frac{1}{C} \bigg( \frac{\log n
}{n} \bigg) ^{1-\frac{1}{\min (p,2)}} \leq K(B_{\ell^n_{p}}) \,.$$
\end{Thm}
The proof  is an almost immediate consequence  of the basic link from \eqref{eq2.1} and
Theorem \ref{mainII}, see section 6. As pointed out above the case $p=\infty$ also follows from  \eqref{Sidon} (which  is itself an immediate consequence
of Theorem \ref{mainresult}, see above).

Let us again emphasize that in Theorem \ref{mainIII} (as in Theorem \ref{mainII}) the most important case $p=\infty$ was observed independently and through a substantially different proof by
Ortega-Cerd{\`a}, Ouna\"ies and Seip in their very recent article \cite[Theorem 2]{OrOuSe}.

\section{More preliminaries}
We use standard notation and notions from (local) Banach space theory, as
presented e.g.\ in \cite{DefFlo}, \cite{DiJaTo95}, \cite{LinTza} or \cite{Tom}.
All considered Banach spaces $X$
are assumed to be complex. We denote their open unit balls by $B_X$ and their duals by $X^*$. The Minkowski spaces
$\ell^n_p$ were already defined in the introduction.

We denote by $ {\rm gl}(X) $
the Gordon-Lewis constant of a Banach space $ X $ (see section 4 for the definition), by $ \lambda(X) $
the projection constant (see section 5 for the definition),  and by $d(X,Y)$
the Banach-Mazur distance between the Banach spaces $ X $ and $Y$.
The $1-$summing norm of a (linear and bounded) operator $T :X \to Y$ is denoted by
$\pi_1(T)$ (we recall this definition in section 3).
A Schauder basis $(x_n)$ of a Banach space $X$ is said to be unconditional
if there is a constant $c\ge 0$ such that$\|\sum^n_{k=1} |\alpha_k|\, x_k\|
\le c\| \sum^n_{k=1} \alpha_k \, x_k\|$ for all $n$ and
$\alpha_1,\ldots,\alpha_n \in \C$. In this case, the best constant $c$ is
denoted by $\chi((x_n))$ and called the unconditional basis constant of $(x_n)$.
Moreover, the infimum over all possible constants $\chi(x_n)$ is the
unconditional basis constant $\chi(X)$ of $X$.
We will often consider Banach spaces  $X = (\C^n, \|\cdot\|)$ such that the
standard unit vectors $e_k$, $1\le k\le n$ form  a 1-unconditional basis. Then the $e_k$'s
also form a  1-unconditional basis of the dual space $X^*$.

For the metric theory of tensor products  we refer to \cite{DefFlo}, and for
the metric theory of symmetric
tensor products and spaces of polynomials to \cite{Dineen} and
\cite{Floret}.
If $X = (\C^n, \|\cdot\|)$ is a
Banach space and $m\in \mathbb{N}$, then $\mathcal{P}(^mX)$ stands for the Banach
space of all $m$-homogeneous polynomials $p(z) = \sum_{|\alpha| = m}
c_\alpha \, z^\alpha,\, z \in \C^n$, together with the norm
$\|p\|_{\mathcal{P}(^mX)} := \sup_{\|z\| \le 1} |p(z)|$. The unconditional basis
constant of all monomials $z^\alpha$, $|\alpha| = m$, is denoted by
$\chi_{\mon}(\mathcal{P}(^mX))$. We identify $ \mathcal{P}(^mX) $ with the space $ \mathcal{L}_s(^m X) $
of symmetric m-linear forms, which is a subspace of $ \mathcal{L}(^m X) $, the space of
m-linear forms. From the polarization formula we get
\[
 \|p\|_{\mathcal{P}(^mX)} \leq
\|p\|_{\mathcal{L}_s(^mX)} \leq \frac{m^m}{m!}\|p\|_{\mathcal{P}(^mX)}.
\]

Sometimes it will be more convenient to think in terms of (symmetric) tensor products
instead of spaces of polynomials.
For a vector space $X$ we denote by
$\otimes^m X$ the $m$th full tensor product, and by $\otimes^{m,s} X$
the  $m$th symmetric tensor product.
Recall
that   $\otimes^{m,s} X$ can be identified with the image of the
symmetrization operator
$$
\begin{array}{lccc}
S_m:&\otimes^m X&{\longrightarrow} &\otimes^{m} X
\\&y_1\otimes ...\otimes y_m & \mapsto &
\frac{1}{m!}\sum_{\sigma \in {\Pi}_m}y_{\sigma (1)}\otimes
....\otimes y_{\sigma(m)},
\end{array} $$
where $\Pi_m$ stands for the group of permutations of
$\{1,...,m\}$; note that the symmetrization operator
in fact is a projector. We will often use the fact that there is some
absolute constant $C \geq 1$ such that  for any $n,m$
\begin{equation} \label{dimension}
   \dim \otimes^{m,s} \C^n = \sum_{ |\alpha|=m} 1
   = \binom{n+m-1}{n-1} \le C^m \Big(1 + \frac{n}{m}\Big)^m\, ;
\end{equation}
this follows by an easy calculation using
Stirling's formula.

Recall
the notation for injective and projective full and symmetric tensor
products of Banach spaces (we
follow  \cite{Floret}): We  write $\otimes^m_\alpha X$ for the $m$th full
tensor product endowed with the  injective norm
$\alpha = \varepsilon$ or projective norm $\alpha = \pi$. Moreover, we write
 $\otimes^{m,s}_{\alpha_s}X$ for the $m$th
symmetric tensor product
of $X$ endowed with the symmetric injective norm
$\varepsilon_s$ or symmetric projective norm $\pi_s$, respectively.
If $\alpha=\varepsilon$ or $\pi$, then by  $ \otimes^{m,s}_\alpha X $ we mean   the $ m$th
symmetric tensor product equipped with  $\alpha$-norm induced by $\otimes^m_\alpha X$.
For $z\in \otimes^{m}X$ we have by the polarization formulas (see e.g.\ \cite[pp.~165,167]{Floret})
\begin{equation} \label{pol1}
   \varepsilon_s (S(z)) \le \varepsilon (S(z)) \le \varepsilon(z) \,\,\, \text{and} \,\,\, \varepsilon (S(z)) \le \frac{m^m}{m!} \varepsilon_s(S(z))\,,
\end{equation}
\begin{equation} \label{pol2}
\pi (S(z)) \le \pi_s (S(z)) \le \frac{m^m}{m!} \pi(z)\,\,\, \text{and} \,\,\, \pi(S(z)) \le \pi_s(S(z))\,.
\end{equation}
The symmetrization operator $ S_m : \otimes^{m}_\alpha X \to \otimes_\alpha^{m} X $ is a norm $1$ projection
onto $\otimes_\alpha^{m,s} X$, and in particular
$ S_m:  \otimes^{m}_\alpha X \to \otimes_{\alpha}^{m} X $ is a projector onto $\otimes_{\alpha_s}^{m,s} X $ of
norm $1$ for $ \alpha = \varepsilon$ and
of norm $\leq \frac{m^m}{m!} $ for $\alpha = \pi$.

Let us fix some useful index sets: For natural numbers
$m,n$ we define $M(m,n):=\{i=(i_1,...,i_m):\ \ i_1,...,i_m
\in\{1,...,n\}\}$  and
$J(m,n):=\!\{j=(j_1,...,j_m)\!\in {M(m,n)}: j_1\leq ...\leq j_m \}$.
We will consider the following
equivalence relation for multi-indices $i,j\in M(m,n)$:
$i \sim j\iff \exists \sigma \in \Pi_m$ such
that $i_{\sigma(k)}=j_k$ for every $k=1,\ldots,m$.
The class of equivalence defined by $i$ is denoted by $[i]$. Also we
denote by
$|i|:={\rm card}[i]$ the cardinal of $[i]$. Note that for each
$i \in M(m,n)$
there is a unique $j\in J(m,n)$ with $[i]=[j]$.
Moreover, for elements $x_1,\dots,x_m$ in a vector space $X$ and $i \in M(m,n)$ define
$ x_{i}:=x_{i_1}\otimes\dots\otimes x_{i_m}\in
\otimes^m X$.
Tn this context the following elementary lemma from  \cite[Lemma 1]{DeDiGaMa} will be used frequently:
\begin{Lem}\label{pony}
Let $m \in \mathbb{N}$ and  $X$ a finite dimensional vector
space  with a basis $(x_k)_{k=1}^n$. Denote the orthogonal basis of the algebraic dual
$X^\ast$ of $X$ by $(x_k^\ast)_{k=1}^n$, i.e $x_l^\ast( x_k)=\delta_{lk}$. Then
$(S(x_{j}))_{j\in J(m,n)}$ is a
basis of $ \otimes^{m,s} X $ and $(|j|S(x^\ast_{j}))_{j\in J(m,n)}$ is
its orthogonal basis in   $\otimes^{m,s}X^\ast$. Moreover, we have
\begin{equation*}
S(x_{i})=\frac{1}{|i|}\sum_{j\in [i]}x_{j} \,\,\, \text{for all} \,\,\, i \in M(m,n)\,.
\end{equation*}
\end{Lem}
There is a one-to-one correspondence between
$J(m,n)$ and $\Lambda(m,n)=\{ \alpha \in \mathbb{N}_{0}^{n} \ \colon | \alpha |=m \}$:
 If $j \in
J(m,n)$ there is an associated multi-index $\alpha$ given by
$\alpha_{r}= |\{ k \colon j_{k} = r \}|$ (i.e. $\alpha_{1}$ is the number of $1$'s in
$j$, $\alpha_{2}$ is the number of $2$'s, \dots), and conversely, if $\alpha \in \Lambda(m,n)$, then
the associated index is given by $j = (1, \stackrel{\alpha_{1}}{\dots} , 1,
2,\stackrel{\alpha_{2}}{\dots} ,2 , \dots ) \in
J(m,n)$. We have $$|j|=m!/\alpha!\,.$$
Moreover, identifying $z^{\alpha}= z^{j_1} \ldots z^{j_m} =S(e_j)$ we have
\begin{eqnarray*}
\chi_{\mon}(\mathcal{P}(^mX)) =
\chi \big( (S(e_j)\big)_{j\in J(m,n)} ; \otimes_{\varepsilon_s}^{m,s}\, X^*\big).
\end{eqnarray*}

Finally we mention the following isometric equalities
 which will be used frequently: For every
finite dimensional Banach space $X$  we have
\begin{equation} \label{duality}
\otimes_\varepsilon X^* = (\otimes_\pi^m X)^* \,\,\,
\text{and} \,\,\, \otimes^{m,s}_{\varepsilon_s} X^* =  (\otimes^{m,s}_{\pi_s} X)^*\,,
\end{equation}
as well as the identifications
\begin{gather*}
   \otimes^{m}_{\varepsilon} X^* = \mathcal{L}(^mX), \,\,
 (x_1^* \otimes \cdots \otimes x_m^*)  \rightsquigarrow [x_1 \otimes \cdots \otimes x_m \rightsquigarrow \prod_k x_k^*(x_k)]\,, \\
   \otimes^{m,s}_{\varepsilon_s} X^*  = \mathcal{L}_s(^mX) = \mathcal{P}(^mX), \,\,
   x^* \otimes \cdots \otimes x^*  \rightsquigarrow [x \rightsquigarrow x^*(x)^m]\, .
\end{gather*}

\section{A fundamental estimate and the proof of Theorem \ref{mainresult}}
\noindent
Recall that the $1$-summing norm of a
linear operator $T: X \rightarrow Y$ (between finite dimensional Banach spaces) is given by
\[\pi_1(T):=\sup\{\sum_{i=1}^n \|Tx_i\|\ : \
\|\sum_{i=1}^n \lambda_ix_i\|\leq 1, n \in \mathbb{N}, |\lambda_i|\leq
1\}\,;
\]
it is well known that
\begin{equation} \label{ten}
\pi_1(T)=\sup_n \|id \otimes T : \ell_1^n \otimes_\varepsilon X
\longrightarrow
\ell_1^n \otimes_\pi Y = \ell^n_1(Y)\|
\end{equation}
(see e.g. \cite{DefFlo} or \cite{DiJaTo95}).
Define for  $m$ and $n$  the canonical mapping
\begin{equation} \label{T}
\begin{array}{lccc}
T:&\mathcal{P}(^m \ell_\infty^n)&{\longrightarrow} &\ell_2^{n+m-1\choose n-1}
\\&\sum_{|\alpha| = m} a_\alpha z^\alpha  & \mapsto &
\big(  a_\alpha \big)_{|\alpha|=m}.
\end{array}
\end{equation}
The fundamental tool of the whole paper  is an estimate for the $1$-summing norm of $T$. The proof is modelled along the proof
of Therorem 3.2 from the phd-thesis of F. Bayart
\cite{Ba} which itself is based on a hypercontractivity result of A. Bonami \cite{Bo}.
\begin{Lem} \label{absolut} For each $m$ and $n$  the operator defined in \eqref{T} satisfies
\[\pi_1(T: \mathcal{P}(^m \ell_\infty^n)\longrightarrow\ell_2^{n+m-1\choose n-1}) \leq \sqrt{2}^m\,.\]
\end{Lem}
\begin{proof}
Let $\mu$ the normalized Lebesgue measure on the torus
$\Tt = \{z\in \C: |z| = 1\}$, and $\mu^n := \otimes_{k=1}^n \mu$ the product measure on the $n$-
dimesinonal torus $\Tt^n$.
It is well known that the $\pi_1$-norm  of the canonical inclusion $L_\infty (\mu^n) \hookrightarrow L_1 (\mu^n)$
equals $1$ (see e.g. \cite{DefFlo} or \cite{DiJaTo95}). Since
$\mathcal{P} (^m\ell_\infty^n)$ is an isometric subspace of $L_\infty (\mu^n)$ (maximum modulus theorem),
it remains to show that for
every $m$-homogeneous
polynomial
$P(z) =\sum\limits_{|\alpha|=m} a_\alpha z^\alpha$
on $\C^n$
\[
\big( \sum\limits_{|\alpha| = m} |a_\alpha|^2 \big)^{\frac{1}{2}} =
\Big\| \D\sum\limits_{|\alpha|=m} a_\alpha z^\alpha\Big\|_{L_2 (\mu^n)}
\leq \sqrt{2}^m \, \Big\| \D\sum\limits_{|\alpha|=m} a_\alpha z^\alpha\Big\|_{L_1 (\mu^n)}
\]
(the first equality is a consequence of the orthogonality of the monomials in $L_2 (\mu^n)$).
Now we follow precisely the proof of F. Bayart \cite[Theorem 3.2.]{Ba}
A result of A. Bonami \cite[Theorem III 7]{Bo} states that for every polynomial
$\sum\limits_{\nu=0}^m a_\nu z^\nu$ in one complex variable\[
\Big\| \D\sum\limits_{\nu=0}^m\, \D\frac{1}{\sqrt{2}^\nu}\, a_\nu z^\nu\Big\|_{L_2(\mu)} \leq
\Big\| \D\sum\limits_{\nu=0}^m \, a_\nu z^\nu\Big\|_{L_1 (\mu)}\,.
\]
But then we conclude with the continuous Minkowski inequality that
\[
\begin{array}{l}
\D\frac{1}{\sqrt{2}^m} \, \left( \D\int\limits_{\Tt^n} \Big| \D\sum\limits_{|\alpha|=m} a_\alpha
  z_1^{\alpha_1}\ldots z_n^{\alpha_n}\Big|^2\, d \mu^n (z)\right)^{\frac{1}{2}} \\ \\

= \left( \D\int\limits_{\Tt^{n-1}}\, \D\int\limits_{\Tt} \,
  \Big| P(\frac{z_1}{\sqrt{2}},\ldots,\frac{z_n}{\sqrt{2}})\Big|^2 \, d\mu (z_n)\, d \mu^{n-1}\,
	(z_1,\ldots,z_{n-1})\right)^{\frac{1}{2}} \\ \\

\leq \left(	\D\int\limits_{\Tt^{n-1}} \left( \D\int\limits_\Tt \,
\Big|
P(\frac{z_1}{\sqrt{2}},\ldots,\frac{z_{n-1}}{\sqrt{2}},z_n)
\Big| \,
d\mu (z_n)\right)^2 \, d\mu^{n-1}\, (z_1,\cdots, z_{n-1})\right)^{\frac{1}{2}}
\\
\leq \D\int\limits_\Tt \, \left( \D\int\limits_{\Tt^{n-1}}
\Big| P(\frac{z_1}{\sqrt{2}},\ldots,\frac{z_{n-1}}{\sqrt{2}}, z_n)\Big|^2 ~
d\mu^{n-1}\, (z_1,\ldots,z_{n-1})\,\right)^{\frac{1}{2}} d\mu (z_n).
\end{array}
\]
The same argument applied to the other coordinates $z_{n-1}, \ldots, z_1$ gives then
as desired
\begin{align*}
\D\frac{1}{\sqrt{2}^m} \,  \left( \D\int\limits_{\Tt^n} \Big| P(z)\Big|^2\,  d \mu^n (z)\right)^{\frac{1}{2}}
  \leq
\,\D\int\limits_{\Tt^n}\Big| P(z)\Big| \, d\mu^n (z)\,.
\end{align*}
\end{proof}
For the proof of  Theorem \ref{mainresult} we will need another lemma due to Blei \cite{Bl}:
For all families $(c_i)_{i \in M(m,n)}$ of complex numbers
\begin{equation} \label{Blei}
\Big(\sum_{i \in M(m,n)} |c_i|^{\frac{2m}{m+1}} \Big)^{\frac{m+1}{2m}}
\leq \prod_{1 \leq k \leq m} \Big(\sum_{i_k = 1}^n \Big(\sum_{i^k \in M(m-1,n)} |c_i|^2 \Big)^\frac{1}{2}\Big)^{\frac{1}{m}}\,;.
\end{equation}
here the following notation is used
$$
\sum_{i^k \in M(m-1,n)} := \sum_{i_1,\ldots,i_{k-1},i_{k+1, \ldots, i_m}= 1}^n
$$
Finally we are  ready to give {\bf the proof of Theorem \ref{mainresult}}:
Again we use the representation $$ \mathcal{P}(^{m} \ell_\infty^n ) = \otimes_{\varepsilon_s}^{m,s} \ell_1^n\,. $$
\noindent Step 1. Define
\[
\begin{array}{lccc}
I:&\otimes^{m,s}_\varepsilon \ell_1^n &{\longrightarrow} &\ell_2 (M(m,n))\,
\\&\sum_{i \in M(m,n)} r_i e_i \,\,,r_i = r_j  \,\,\,\text{for} \,\, [i]= [j] & \mapsto &
(\sqrt{|i|} r_i)_{i \in M(m,n)}\,.
\end{array}
\]
We show that  $\pi_1(I) \leq \sqrt{2}^{m} $. Indeed, from the preceding lemma and \eqref{pol1}
we get that the $ \pi_1$-norm of the map
\[
\begin{array}{lccc}
T :& \otimes^{m,s}_\varepsilon \ell_1^n
&{\longrightarrow} &
\ell_2(J(m,n))
\\&
\sum_{j \in J(m,n)} \lambda_j S(e_j)
& \mapsto &
(\lambda_j)_{j \in J(m,n)}.
\end{array}
\]
is $ \leq \sqrt{2}^{m} $. Note that for
$$ \sum_{i \in M(m,n)} r_i e_i \in
\otimes^{m,s}\ell_1^n
 ~ \,\,\,\mbox{with}\,\,\, r_i = r_j \mbox{ for}\,\,\, [i]= [j], $$
we have (see \ref{pony})
$$
\sum_{i \in M(m,n)} r_i e_i = \sum_{j \in J(m,n)} |j|r_j S(e_j),
$$
hence the $\pi_1$-norm of (the same) map
\[
\begin{array}{lccc}
K:&
 \otimes^{m,s}_\varepsilon \ell_1^n
&{\longrightarrow} &
\ell_2(J(m,n))
\\&
\sum_{i \in M(m,n)} r_i e_i \,\,,r_i = r_j  \,\,\,\text{for} \,\, [i]= [j]
& \mapsto &
(|j|r_j)_{j \in J(m,n)}
\end{array}
\]
is $ \leq \sqrt{2}^m $. Now consider

\[
\begin{array}{lccc}
J:&
 \ell_2(J(m,n))
&{\longrightarrow} &
\ell_2(M(m,n))
\\&
(\lambda_j)_{j \in J(m,n)}
& \mapsto &
\big(\frac{\lambda_j}{\sqrt{|j|}} \big)_{i \in [j],~ j \in J(m,n)}\,.
\end{array}
\]
Then $ J $ is an isometry. Since $I = J \circ K$ we obtain as desired $ \pi_1(I) \leq \sqrt{2}^{m}$.
\\
Step 2. We show that
\begin{align*}
\sum_{i_k = 1}^n \Big( \sum_{i^{k} \in M(m,n)}  (\sqrt{|(i_0,i_1,\ldots,i_m)|} &
|\lambda_{(i_0,i_1,\ldots,i_m)}|)^2 \Big)^{\frac{1}{2}} \\ &\leq \sqrt{2}^m \sqrt{m+1}
\,\,\varepsilon \big( \sum_{i \in M(m+1,n)}  \lambda_i e_i \big),
\end{align*}
where  $0 \leq k \leq m$ and
$(\lambda_i)_{i \in  M(m+1,n)}$ is a family of complex numbers for which
$ \lambda_{(i_0,i_1,\ldots,i_m)} = \lambda_{(j_0,j_1,\ldots,j_m)}$ for $[(i_0,i_1,\ldots,i_m)]= [(j_0,j_1,\ldots,j_m)].$
>From now on we denote $ (i_0,i_1,\ldots,i_m) =: (i_0,i) \in M(m+1,n) $, and have hence $\lambda_{(i_0,i)} = \lambda_{(k_0,k)}$
if $[(i_0,i)]=[(k_0,k)] $.
If we consider  $\otimes_\varepsilon^{m+1,s} \ell_1^n $ as a subspace of $ \ell_1^n \otimes_\varepsilon \otimes_\varepsilon^{m,s} \ell_1^n$,
then Step 1 and \eqref{ten} imply that the (operator) norm of the mapping
\[\otimes_{\varepsilon}^{m+1,s}\hookrightarrow \ell_1^n \otimes_\varepsilon
(\otimes_{\varepsilon}^{m,s} \ell_1^n) \to \ell_1^n \otimes_\pi \ell_2 (M(m,n)) = \ell_1^n (\ell_2(M(m,n)))\,,\]
which assigns to every
\[
 \sum_{(i_0,i) \in M(m+1,n)} \lambda_{(i_0,i)} e_{(i_o,i)}
= \sum_{i_0 = 1}^n e_{i_0} \otimes \sum_{i \in M(m,n)} \lambda_{(i_0,i)} e_i
\in \ell_1^n \otimes (\otimes_{\varepsilon}^{m,s} \ell_1^n)
 \]
the element
\[\big( (\sqrt{|i|} \lambda_{(i_0,i)})_{i \in M(m,n)}\big)_{1 \leq i_0 \leq n} \in \ell^n_1(\ell_2(M(m,n)))\,,
\]
is $\leq \sqrt{2}^{m} $.
But this means precisely that
$$
\sum_{i_0 = 1}^n \Big( \sum_{i \in M(m,n)} (\sqrt{|i|} |\lambda_{(i_0,i)}|)^2 \Big)^{\frac{1}{2}} \leq \sqrt{2}^{m}
\varepsilon \big( \sum_{(i_0,i) \in M(m+1,n)}  \lambda_{(i_0,i)} e_{(i_0,i)} \big).
$$
Since $$\frac{|(i_0,i)|}{|i|} = \frac{m+1}{|\{ \nu \in \{1,\ldots, m \} :  i_\nu = i_0  \}| +1} \leq m+1 \,,$$
we get
\begin{align*}
\sum_{i_0 = 1}^n \Big( \sum_{i \in M(m,n)} (\sqrt{|(i_0,i)|} & |\lambda_{(i_0,i)}|)^2 \Big)^{\frac{1}{2}} \\ & \leq \sqrt{2}^{m} \sqrt{m+1}
\,\,\varepsilon \big( \sum_{(i_0,i) \in M(m+1,n)}  \lambda_{(i_0,i)} e_{(i_0,i)} \big).
\end{align*}
Clearly,  we can apply this inequality also to the other coordinates $i_1,\ldots,i_m$, and hence we obtain as desired
for all  $0 \leq k \leq m $ and all $\lambda_i \in M(m+1,n)$  with $\lambda_i 0 \lambda_j $
for $[i]= [j]$
that
$$\sum_{i_k = 1}^n \Big( \sum_{i^{k} \in M(m,n)} (\sqrt{|i|} |\lambda_{i}|)^2 \Big)^{\frac{1}{2}} \leq \sqrt{2}^{m} \sqrt{m+1}
\,\,\varepsilon \big( \sum_{i \in M(m+1,n)}  \lambda_i e_i \big). $$ \\
Step 3. Blei's inequality \eqref{Blei} applied to preceding inequality from Step 2 (for $ m-1 $ instead of $m$) implies that
$$
\big(\sum_{ i \in M(m,n)} (\sqrt{|i|}|\lambda_{i}|)^{\frac{2m}{m+1}} \big)^{\frac{m+1}{2m}} \leq \sqrt{2}^{m-1} \sqrt{m} \,\,
\varepsilon \big( \sum_{ i \in M(m,n)}  \lambda_{ i } e_{ i } \big).
$$
Step 4. Finally we show for all families $(\lambda_j)_{j \in J(m,n)}$ of complex numbers that
$$ \big(\sum_{j\in J(m,n)}  |\lambda_j|^{\frac{2m}{m+1}}\big)^{\frac{m+1}{2m}} \leq
\sqrt{2}^{m-1}\sqrt{m} \varepsilon \big( \sum_{j \in J(m,n)}  \lambda_j S(e_j)   \big)\,,
$$  and this finishes the proof of Theorem \ref{mainresult}: From Step 3 applied to $ \tilde{\lambda_i} :=
\frac{\lambda_j}{|j|}, ~ i \in [j] $, we obtain

\begin{align*}
 \big(\sum_{j\in J(m,n)}  |\lambda_j|^{\frac{2m}{m+1}} \big)^{\frac{m+1}{2m}}  &  =
 \big(\sum_{j\in J(m,n)} \sum_{i \in [j]} |j|^{\frac{-1}{m+1}}
( \sqrt{|j|}\frac{|\lambda_j|}{|j|})^{\frac{2m}{m+1}}\big)^{\frac{m+1}{2m}}   \\ \\
&\leq
 \big(\sum_{j\in J(m,n)} \sum_{i \in [j]}
( \sqrt{|j|}\frac{|\lambda_j|}{|j|})^{\frac{2m}{m+1}}\big)^{\frac{m+1}{2m}}   \\ & =
\big(\sum_{i\in M(m,n)}
( \sqrt{|i|}\tilde{|\lambda_i|})^{\frac{2m}{m+1}}\big)^{\frac{m+1}{2m}}   \\ & \leq
\sqrt{2}^{m-1} \sqrt{m} ~ \varepsilon \big( \sum_{i \in M(m,n)} \tilde{\lambda_i} e_i  \big) \\ &=
 \sqrt{2}^{m-1} \sqrt{m}   ~  \varepsilon \big( \sum_{j \in J(m,n)} \lambda_j S( e_j) \big)
\\ &\leq
\sqrt{2} ^{m-1} \sqrt{m} \frac{m^m}{m!}  ~  \varepsilon_s \big( \sum_{j \in J(m,n)} \lambda_j S( e_j) \big)
\end{align*}
Since there obviously is some constant $C \geq 1$ such that
 $\sqrt{2}^{m-1} \sqrt{m} \frac{m^m}{m!} \leq C^m$ for all $m$, the proof is complete. $\,\, \Box$

\section{Gordon-Lewis and unconditional basis constants}
A  Banach space invariant  very closely related to  unconditional basis constants
is the
Gordon-Lewis constant invented in the classical paper \cite{GOLE}. A
Banach
space $X$ is said to  have
the Gordon-Lewis property if every 1-summing operator
$T:X\longrightarrow \ell_2$
allows a factorization
$ T:X \stackrel{R}{\longrightarrow}L_1(\mu)
\stackrel{S}{\longrightarrow}\ell_2$
($\mu$ some measure, $R$ and $S$ operators). In this case, there is a
constant $c\geq0$ such that
$\gamma_1(T):=\inf \|R\| \|S\|\leq c\pi_1(T)$ for all $T:X\longrightarrow \ell_2$,
and the best such $c$ is called the Gordon-Lewis constant of $X$ and
denoted
by $\mbox{gl}(X)$. We are going to use the obvious fact that for two Banach spaces $X,Y$
\begin{equation}\label{golemmazur}
\mbox{gl}(X) \leq d(X,Y) \mbox{gl}(Y)\,.
\end{equation}

A fundamental tool for the study of unconditionality in Banach spaces is
the
Gordon-Lewis inequality from \cite{GOLE} (see also
\cite[17.7]{DiJaTo95}): For every unconditional basis  $(x_i)$ of
a (complex) Banach space $X$ we have
\begin{equation}\label{golem}
\mbox{gl}(X) \leq 2 \chi((x_i))\,.
\end{equation}
We now follow a cycle of
ideas invented in \cite{PI78,Schutt} and which was later applied to spaces of $m$-homogeneous polynomials in \cite{DeDiGaMa}.
Given a Banach space $X_n = (\C^n, \|\cdot\|)$ for which
the $e_k$'s form a $1$-unconditional basis,
for Banach spaces $\mathcal{P}(^m X_n)$ the converse of the Gordon-Lewis inequality holds true;
the main difference to \cite[Theorem 1]{DeDiGaMa} is the hypercontractivity  the constant.

\begin{Prop} \label{maingole}
There are constants $C \geq 1$ such that for  each Banach space $X_n = (\C^n, \|\cdot\|)$ for which
the $e_k$'s form a $1$-unconditional basis, we
$$
\chi_{\mon}(\mathcal{P}(^m X_n)) \leq C^m {\rm gl} (\mathcal{P}(^m X_n)).
$$
\end{Prop}
We prefer to prove this result in terms of symmetric tensor products; again we use the representation $ \mathcal{P}(^{m} X_n ) = \otimes_{\varepsilon_s}^{m,s} X_n^* $ (see \eqref{duality}).
In the following $\alpha$ will always be either the projective tensor norm $\pi$ or the
injective tensor norm $\varepsilon$, and $\alpha_s$ stands either for the symmetric  projective tensor norm $\pi_s$ or the
symmetric injective tensor norm $\varepsilon_s$. Moreover, we put $\pi^* = \varepsilon$ and  $\varepsilon^* = \pi$, as well as
$\pi^*_s = \varepsilon_s$ and  $\varepsilon^*_s = \pi_s$ (see \eqref{duality}). The following result is a reformulation of the preceding one with a more precise constant.

\begin{Prop} \label{Gordon-Lewis-Equi}
Let  $X$  be a Banach space with the 1-unconditional basis
$(x_k)_{k=1}^n$, and let $ \alpha_s $ be either $\pi_s$ or $\varepsilon_s$.
Then
\[
\begin{array}{l}
\chi \big( (S(x_i)\big)_{i\in J(m,n)} ; \otimes_{\alpha_s}^{m,s}\, X\big)
\leq \left( \D\frac{m^m}{m!}\right)^2\, 2^m {\rm gl} \, ( \otimes_{\alpha_s}^{m,s}\, X).
\end{array}
\]
\end{Prop}
Again we devide the proof into several steps. The first is \cite[Lemma 4]{DeDiGaMa} which we repeat for the
sake completeness.

\begin{Lem} \label{JFA}
 Let $Y$ be a finite dimensional Banach
space with a basis $(y_i)_{i=1}^{n}$
and orthogonal  basis $(y^{\ast}_{j})_{j=1}^{n}$.
Suppose that there exist constants $M_1,M_2 \geq 1$ such that for every choice of
$\lambda,\mu \in \C^n$ the diasgonal mappings
$$
\begin{array}{clc}
\begin{array}{lccc}
D_\lambda:&Y
&{\longrightarrow} &\ell_2^{n}
\\&\sum_{i=1}^n  a_{i}y_{i}& \mapsto & (\lambda_{i}a_{i})_{i=1}^n
\end{array}
&\makebox[-4mm]{}&
\begin{array}{lccc}
D_\mu:&Y^\ast
&{\longrightarrow} &\ell_2^{n}
\\&\sum_{j=1}^n  a_{j}y^\ast_{j}& \mapsto & (\mu_{j}a_{j})_{j=1}^n
\end{array}
\end{array}
$$
satisfy
\begin{equation*}
\begin{array}{clc}
\pi_1(D_\lambda)\leq  M_1\|\sum_{i=1}^n \lambda_{i}y^\ast_{i}\|_{Y^\ast}
&,\makebox[2mm]{}
\pi_1(D_\mu)\leq  M_2\|\sum_{j=1}^n\mu_{j}y_{j}\|_Y\,.
\end{array}
\end{equation*}
Then
\[{\chi}((y_i))\leq M_1M_2\,gl(Y).\]
\end{Lem}

The next four lemmata show how to control these diagonal operators in case of symmetric tensor products/spaces
of $m$-homogeneous polynomials.

\begin{Lem} \label{Ab1}
Let  $X$ be a Banach space, and $(x_k)_{k=1}^n$ a 1-unconditional basis.
Then we have for all families  $(\widetilde{c}_j)_{j\in M(m,n)}$ of complex numbers that
the diagonal operator
\[
\begin{array}{lccc}
D_{\widetilde{c}}:&
\otimes_\alpha^m X
&{\longrightarrow} &
\mathcal{L} (^m \ell_\infty^n)
\\&
x_j
& \mapsto &
\big\{ (z^{(1)},\ldots,z^{(m)}) \mapsto \widetilde{c}_j \, z_{j_1}^{(1)}\cdot \ldots \cdot z_{j_m}^{(m)} \big\}
\end{array}
\]
has (operator) norm $\leq \alpha^\ast \big(\sum\limits_{j\in M(m,n)} \widetilde{c}_j x_j^\ast\big)$.
\end{Lem}
\begin{proof}
Define for $ z = (z^{(1)},\ldots,z^{(m)}) \in B_{\ell_\infty^n}^m$
\[
\begin{array}{lccc}
T_z := \otimes_{k=1}^m T_{z^{(k)}}:&
\otimes_\alpha^m X
&{\longrightarrow} &
\otimes_\alpha^m X
\\&
x_{j_1} \otimes \ldots \otimes x_{j_m}
& \mapsto &
(z_{j_1}^{(1)} x_{j_1}) \otimes \ldots \otimes (z_{j_m}^{(m)} x_{j_m})\,.
\end{array}
\]
Since  $(x_k)_{k=1}^n$ is a 1-unconditional basis, we know that $ \| T_z \| \leq 1\,. $
But then we obtain with the mapping property of $\alpha$  for all $z^{(1)},\ldots,z^{(m)} \in B_{\ell_\infty^n}$ that
\begin{align*}
\Big| D_{\widetilde{c}}  \Big( \D\sum\limits_{j\in M(m,n)}&  \widetilde{\lambda}_j x_j\Big) (z^{(1)},\ldots,z^{(m)})\Big|
= \Big| \D\sum\limits_{j\in M(m,n)} \widetilde{\lambda}_j \widetilde{c}_j z_{j_1}^{(1)} \cdot \ldots \cdot z_{j_m}^{(m)}\Big| \\\\&
\leq  \alpha\, \Big( \D\sum\limits_{j\in M(m,n)} (\widetilde{\lambda}_j z_{j_1}^{(1)}\cdot \ldots z_{j_m}^{(m)})\, x_{j} \Big)
\,\alpha^\ast\, \Big( \D\sum\limits_{j\in M(m,n)} \widetilde{c}_j x_{j}^{ \ast}\Big) \\ &
= \alpha\, \Big( T_z (\D\sum\limits_{j\in M(m,n)} \widetilde{\lambda}_j x_j)\Big) \,
     \alpha^\ast \, \Big( \D\sum\limits_{j\in M (m,n)} \widetilde{c}_j x_j^\ast\Big)\\&
     \leq \alpha\, \Big(\D\sum\limits_{j\in M(m,n)} \widetilde{\lambda}_j x_j) \,
     \alpha^\ast \, \Big( \D\sum\limits_{j\in M (m,n)} \widetilde{c}_j x_j^\ast\Big)\,,
\end{align*}
which clearly implies as desired that $\| D_{\widetilde{c}}  \| \leq \alpha^\ast \, \Big( \D\sum\limits_{j\in M (m,n)} \widetilde{c}_j x_j^\ast\Big)\,.$
\end{proof}
We proceed with a symmetric version of this lemma.
\begin{Lem} \label{Ab2}
Let $X$ be a Banach space with a 1-unconditional basis $(x_k)_{k=1}^n$. Then for every family
$(c_i)_{i\in J(m,n)}$ of complex numbers  the diagonal operator
\[
\begin{array}{lccc}
D_c:&
\otimes_{\alpha_s}^{m,s} X
&{\longrightarrow} &
\mathcal{L}_s (^m \ell_\infty^n)
\\&
 S(x_i)
& \mapsto &
\big\{ (z^{(1)},\ldots,z^{(m)})  \mapsto   c_i \D\frac{1}{|i|} \D\sum\limits_{j\in[i]}\,
z_{j_1}^{(1)}\cdot \ldots \cdot z_{j_m}^{(m)} \big\}
\end{array}
\]
has norm
$\leq  \D\frac{m^m}{m!} \, \alpha_s^\ast \, \Big( \sum\limits_{i\in J(m,n)} c_i |i| S(x_i^\ast)\Big)$.
\end{Lem}
\begin{proof}
Take $\sum\limits_{i \in J(m,n)} \lambda_i S(x_i) \in \otimes^{m,s} X$ , and apply the preceding Lemma
to  $ \widetilde{\lambda}_j := \frac{\lambda_i}{|i|} $ and $\widetilde{c}_j := c_i, j \in [i]$. Then
\begin{align*}
\Big| \big[ D_c  \big( \D\sum\limits_{i\in J(m,n)}&  \lambda_i S(x_i)\big]\, (z^{(1)},\ldots,z^{(m)})\Big| \\\\[-3ex] &
 = \Big| \D\sum\limits_{i\in J(m,n)} \lambda_i c_i \, \D\frac{1}{|i|} \, \D\sum\limits_{j\in[i]}
      z_{j_1}^{(1)} \cdot \ldots \cdot z_{j_m}^{(m)}\Big| \\&
      =  \Big| \D\sum\limits_{i\in J(m,n)} \, \D\sum\limits_{j\in[i]}\, \D\frac{\lambda_i}{|i|}\,
      c_i z_{j_1}^{(1)} \cdot \ldots \cdot z_{j_m}^{(m)}\Big| \\ &
          \leq\, \alpha \, \Big( \D\sum\limits_{i\in J(m,n)}\,  \D\sum\limits_{j\in[i]}
      \D\frac{ \lambda_i }{|i|}\, x_j\Big)  \alpha^\ast \Big( \D\sum\limits_{i\in J(m,n)} \, \D\sum\limits_{j\in[i]}\,
      c_i x_j^\ast\Big) \\&  =\, \alpha \, \Big( \D\sum\limits_{i\in J(m,n)} \lambda_i S(x_i)\Big) \,
      \alpha^\ast\, \Big( \D\sum\limits_{i\in J(m,n)} c_i |i| S(x_i^\ast)\Big) \\&
          \leq \,  \D\frac{m^m}{m!} \, \alpha_s \, \Big( \D\sum\limits_{i\in J(m,n)} \lambda_i S(x_i)\Big)\,
          \alpha_s^\ast \, \Big( \D\sum\limits_{i\in J(m,n)} c_i |i| S(x_i^\ast)\Big)\,,
\end{align*}
where the latter inequality follows from \eqref{pol1} and \eqref{pol2}.
\end{proof}
The last lemma  needed for the proof of Proposition \ref{Gordon-Lewis-Equi} is an immediate consequence of the preceding one and our fundamental estimate from Lemma \ref{absolut}.
\begin{Lem} \label{Ab3}
Let $(x_k)_{k=1}^n$ be a 1-unconditional basis of the Banach space $X$. Then for every family
$(c_i)_{i\in J(m,n)}$ of complex numbers the diagonal operator
\[
\begin{array}{lccc}
D_c:&
\otimes_{\alpha_s}^{m,s}\, X
&{\longrightarrow} &
\ell_2 \big( J(m,n)\big)
\\&
\D\sum\limits_{i\in J(m,n)} \lambda_i S(x_i)
& \mapsto &
(\lambda_i c_i)_{i\in J(m,n)}
\end{array}
\]
satisfies
\[
\pi_1 \, (D_c) \leq  \D\frac{m^m}{m!}\, \sqrt{2}^m \, \alpha_s^\ast\,
\Big( \D\sum\limits_{i\in J(m,n)} c_i |i| \, S(x_i^\ast)\Big).
\]
\end{Lem}

Note now finally that Lemma \ref{JFA}, the preceding Lemma \ref{Ab3} and Lemma \ref{pony} together yield
\[
\begin{array}{l}
\chi \big( (S(x_i)\big)_{i\in J(m,n)} ; \otimes_{\alpha_s}^{m,s}\, X\big)
\leq \left( \D\frac{m^m}{m!}\right)^2\, 2^m {\rm gl} \, ( \otimes_{\alpha_s}^{m,s}\, X)\,,
\end{array}
\]
which  completes {\bf the proof of Proposition \ref{Gordon-Lewis-Equi}} (which was nothing else than
a tensor product formulation of the main result of this section, Proposition \ref{maingole}).

\section{Gordon-Lewis constants and projection constants}

Recall that the projection constant of a finite dimensional Banach space $X$ is defined to be
\[
\lambda (X) = \sup\{\lambda(I(X),Z )\,:\, I: X\hookrightarrow Z \;\; \text{an
isometric embedding into} \;\;Z\}\,,
\]
where for a subspace $Y$ of a Banach space $Z$ the relative projection constant
$\lambda(Y,Z)$ is the infimum of all $\|P\| $ taken with respect to all projections
$P$ onto $Z$. We  will use  the well known estimates (see \cite[9.12]{Tom}).
\begin{equation} \label{dsds}
\lambda (X) \leq \sqrt{\dim X}\,,
\end{equation}
and also the obvious fact that
\begin{equation} \label{dsdsds}
\lambda (X) \leq d(X,Y)\, \lambda (Y)\,,
\end{equation}
The main purpose of this section is to prove the following proposition which in combination with
Proposition \ref{maingole} allows to estimate unconditional basis constants of symmetric tensor products/spaces of $m$-
homogeneous polynomials.

\begin{Prop} \label{golproj}
Let $ X $ be a Banach space  with a 1-unconditional basis $(x_k)_{k=1}^n$.
Then for every $ m \geq 2$ we have
\begin{itemize}
\item[(1)]
$ {\rm gl} \big( \otimes_{\varepsilon}^{m,s} X \big) \leq 2  \lambda \big(  \otimes^{m-1,s}_\varepsilon X  \big) $
\item[(2)]
$ {\rm gl} \big( \otimes_{\varepsilon_s}^{m,s} X \big) \leq 2 \big(\frac{m^m}{m!}\big)^2 \lambda
\big(  \otimes^{m-1,s}_{\varepsilon_s} X  \big)\,. $
\end{itemize}
\end{Prop}

Note that the projection constant of the polynomials appears with degree $m-1$ whereas the Gordon-Lewis constant
is taken with respect to all polynomials of degree $m$. The trick which makes this possible is isolated in the
following lemma.

\begin{Lem} \label{mainproj}
Let $X $ be a finite dimensional Banach space and $m \in \mathbb{N}$. Then

\[
{\rm gl} \big( \otimes_{\varepsilon}^{m+1,s} X \big) \leq \sup_N{\rm gl}
\big(X \otimes_\varepsilon \ell_\infty^N ) \lambda \big(  \otimes^{m,s}_\varepsilon X  \big).
\]

\end{Lem}
\begin{proof}
Step 1.  Let $\varepsilon > 0 $ be arbitrary.
We map $ \otimes^{m,s}_\varepsilon X $ onto a subspace $ Y $ of $\ell_\infty^N$ with
$ d(\otimes^{m,s}_\varepsilon X , Y) \leq 1 + \varepsilon $ such that there is a projector
$P : \ell_\infty^N \to \ell_\infty^N $ onto this subspace with $\| P \|
\leq \lambda (Y) + \varepsilon $. Then
$$ id \otimes P : X \otimes_\varepsilon \ell_\infty^N  \to X \otimes_\varepsilon  Y $$
is a projector with the same norm. Hence $ {\rm gl} (X \otimes_\varepsilon Y ) \leq \| P \| {\rm gl}
\big(X \otimes_\varepsilon \ell_\infty^N ) $. Since $\varepsilon > 0 $ was arbitrary,
we get $$ {\rm gl} \big(X \otimes_\varepsilon (\otimes_\varepsilon^{m,s} X) \big) \leq
 \lambda \big(   \otimes^{m,s}_\varepsilon X     \big)\sup_N {\rm gl} \big(X \otimes_\varepsilon \ell_\infty^N \big).$$
\\
Step 2.
Since the injective norm respects isometric subspaces, $ X \otimes_\varepsilon
(\otimes_\varepsilon^{m,s} X) $ is an isometric subspace of
$ X \otimes_\varepsilon (\otimes_\varepsilon^m X) = \otimes_\varepsilon^{m+1} X$.
Because of
\begin{align*}
\otimes^{m+1, s} X  = & \,\,\text{span}\{ \otimes^{m+1} x : x \in X   \} \\\\&
\subset \text{span}\{ y\otimes (\otimes^m x)
 : y,x \in X  \} = X \otimes (\otimes^{m,s} X)\,,
\end{align*}
we see that $ \otimes_\varepsilon^{m+1,s} X$ is an isometric subspace of
$X \otimes_\varepsilon (\otimes^{m,s}_\varepsilon X) $. Consider now the norm $1$ projection
$$ S_{m+1} :  \otimes_\varepsilon^{m+1} X \to \otimes_\varepsilon^{m+1} X $$
onto $\otimes_\varepsilon^{m+1,s} X$ (see \eqref{pol1}). Clearly, if this map is restricted to  $ X \otimes_\varepsilon (\otimes^{m,s}_\varepsilon X)$, then
we obtain a norm $1$ projection $  X \otimes_\varepsilon (\otimes^{m,s}_\varepsilon X) \to  X \otimes_\varepsilon(\otimes^{m,s}_\varepsilon X) $
onto $\otimes_\alpha^{m+1,s} X $. This finally implies
$$ {\rm gl} \big( \otimes_{\varepsilon}^{m+1,s} X \big)
\leq {\rm gl} \big( X \otimes_\alpha (\otimes_{\varepsilon}^{m,s} X) \big) $$
which together with Step 1 leads to the conclusion.
\end{proof}

Now {\bf the proof of Proposition \ref{golproj}} is easy:
The unconditional basis constant of  $X \otimes_\varepsilon \ell_\infty^N $ is $1$ (see e.g. \cite[Lemma 5]{Schutt}), hence the Gordon-Lewis constant
of this space is $\leq 2$ by \eqref{golem}. To get the first inequality we apply the preceding Lemma.  For the second inequality
recall that we have $ d(\otimes_{\varepsilon_s}^{m,s} X,\otimes_{\varepsilon}^{m,s} X ) \leq \frac{m^m}{m!} $
(see \eqref{pol1}). Hence we obtain from \eqref{golemmazur} and \eqref{dsdsds} that
\begin{align*}
{\rm gl} \big( \otimes_{\varepsilon_s}^{m,s} X \big) & \leq
 \dfrac{m^m}{m!} {\rm gl} \big( \otimes_{\varepsilon}^{m,s} X \big)  \\\\[-3ex]& \leq
2\, \dfrac{m^m}{m!} \lambda \big( \otimes_{\varepsilon}^{m-1,s} X \big)   \\& \leq
2 \,\dfrac{m^m}{m!} \dfrac{(m-1)^{m-1}}{(m-1)!}\lambda \big( \otimes_{\varepsilon_s}^{m-1,s} X \big)  \\ &\leq
2 \,\big(\dfrac{m^m}{m!}\big)^2\lambda \big( \otimes_{\varepsilon_s}^{m-1,s} X \big)\,. \,\,\,\,\,\Box
\end{align*}

\vspace{3mm}

We remark that we already here get an alternative proof of Theorem \ref{mainII} in the case  $p \geq 2$
(recall hat this case was
already proved on the basis of Theorem\ref{mainresult}):
By the propositions \ref{Gordon-Lewis-Equi} and \ref{golproj} as well as
\eqref{dsds} and \eqref{dimension} we have that
\[
\begin{array}{l}
\chi \big( (S(x_j)\big)_{j\in J(m,n)} ; \otimes_{\varepsilon_s}^{m,s}\, X\big)
\leq C^m \lambda ( \otimes_{\varepsilon_s}^{m-1,s}\, X)  \leq C^m (1+\frac{n}{m-1})^{\frac{m-1}{2}}\,.
\end{array}
\]
Hence after identifying $\mathcal{P}(^m \ell_p^n)= \otimes^{m,s}_{\varepsilon_s} \ell_q^n$ with $1/p+1/q =1$,
we conclude
\begin{equation*} \label{klar}
\chi_{\mon} (\mathcal{P}(^m \ell_p^n) \leq C^m \lambda (\mathcal{P}(^{m-1} \ell_p^n) \leq C^m
(1+\frac{n}{m-1})^{\frac{m-1}{2}}\,,
\end{equation*}
the statement of Theorem \ref{mainII} in the case  $p \geq 2$\,.
But for $p \leq 2$ this estimate has to be improved,
and we established this in the two final results of this section.
\begin{Lem}
For a given Banach space  $ X := (\C^n,\|\cdot\|) $ define for each $|\alpha| = m$
$$ d_\alpha := \sup \{|a_\alpha| :  \sup_{z \in B_{X}} | \sum_{|\beta|=m} a_\beta z^\beta |  \leq 1 \}\,. $$
Then $$\lambda( \mathcal{P}(^m X)) \leq \sup_{\| z \|_{X} \leq 1} \sum_{|\alpha|=m} d_\alpha |z^\alpha | =: p(X).$$
\end{Lem}
\begin{proof}
Consider $ \mathcal{P}(^m X) $ as a subspace of  $ \ell_\infty (B_X) $. We construct a projector $ P : \ell_\infty (B_X) \to \mathcal{P}(^m X) $
with norm $ \leq p(X) $.
We use that the functionals  $ k_\alpha : \mathcal{P}(^m X) \to \C, ~ \sum_{|\beta|=m} a_\beta z^\beta \mapsto a_\alpha $ have norm $ d_\alpha $. With the Hahn-Banach theorem we extend them to  $ K_\alpha : \ell_\infty(B_X) \to \C $
with the same norm. \\
Let now $$ P  : \ell_\infty (B_X) \to \mathcal{P}(^m X), ~ f \mapsto \sum_{|\alpha|=m} K_\alpha(f) z^\alpha.$$
Then  $ P $ is a projector on  $ \mathcal{P}(^m X) $ and we have
\begin{eqnarray*} \| P(f) \|_\infty = \| P(f) \|_{P(^m X)} & = & \sup_{\| z \|_{X} \leq 1} | \sum_{|\alpha|=m} K_\alpha(f) z^\alpha | \\
& \leq &  \sup_{\| z \|_{X} \leq 1}  \sum_{|\alpha|=m} |K_\alpha(f)|| z^\alpha | \\
&\leq & \sup_{\| z \|_{X} \leq 1}  \sum_{|\alpha|=m} d_\alpha \|f \|_\infty | z^\alpha | \leq
\|f \|_\infty ~ p(X).
\end{eqnarray*}
\end{proof}

We now follow the proof of \cite[Lemma 3.3]{DeFr} in order to get the needed estimate for the projection constant of
$\mathcal{P}(^m \ell_p^n)$.

\begin{Prop} \label{proproj}
There is a $ C \geq 1 $ such that for all $ 1 \leq p \leq \infty$, all $ n, m \in \mathbb{N} $ we have
$$\lambda(\mathcal{P}(^m \ell_p^n))
\leq C^m (1 + \frac{n}{m})^{m( 1-\frac{1}{\min \{p,2\} })}\,. $$
\end{Prop}
\begin{proof}
The case $p \geq 2$ was already  proved, see \eqref{klar} or the remark after \eqref{Sidon}. For $ 1 \leq p \leq 2$ we apply the preceding Lemma to $ X = \ell_p^n $.
>From the proof of \cite[Lemma 3.3]{DeFr} we know that $ d_\alpha \leq    e^{m/p}\Big(\frac{m!}{\alpha!}\Big)^{1/p}\,,  $
and hence by H\"older's inequality (with $\frac{1}{p} + \frac{1}{q} = 1 $) and \eqref{dimension} for $z\in \C^n$
\begin{align*}
   \sum_{|\alpha|=m} d_\alpha |z^\alpha|
    & \le\; e^{\frac{m}{p}} \sum_{|\alpha|=m}
    \Big(\frac{m!}{\alpha!}\Big)^{\frac{1}{p}}
         |z^\alpha|\\[1ex]
    & \le\; e^{\frac{m}{p}} \Big(\sum_{|\alpha|=m} 1\Big)^{\frac{1}{q}}
         \Big(\sum_{|\alpha|=m} \; \frac{m!}{\alpha!}\;
         (|z_1|^p,\ldots,|z_n|^p)^\alpha\Big)^{\frac{1}{p}}\\[1ex]
    & \le \; e^{\frac{m}{p}} \; c^{\frac{m}{q}} \Big(
         1+\frac{n}{m} \Big)^{\frac{m}{q}} \Big(\sum^n_{k=1}|z_k|^p
         \Big)^{\frac{m}{p}}\\[1ex]
    & =\; e^{\frac{m}{p}} \; c^{\frac{m}{q}} \Big( 1+ \frac{n}{m}
         \Big)^{\frac{m}{q}} \|z\|^m_{\ell^n_p}\, .
\end{align*}
Thus
\[
   \lambda (\mathcal{P}(^m\ell^n_p)) \;
   \le \; e^{\frac{m}{p}} c^{\frac{m}{q}} \Big(1 + \frac{n}{m}\Big)^{\frac{m}{q}}.
\]
\end{proof}

\section{Proofs of Theorem \ref{mainII} and Theorem \ref{mainIII}}
All we have to do is to collect the results already shown in the preceding sections.

\vspace{2mm}

\noindent {\bf (1) Proof of Theorem \ref{mainII}}:
Fix $m,n$ and $1 \leq p \leq \infty$.
We again identify $\mathcal{P}(^m\ell_p^n) =\otimes_{\varepsilon_s}^{m,s}\, \ell_q^n\,$, where $1/p+ 1/p = 1$\,.
>From Proposition \ref{Gordon-Lewis-Equi} we know that
$$
\chi_{\mon}(\mathcal{P}^m(\ell_p^n))  \big)
\leq C^m {\rm gl} \, (\mathcal{P}^m(\ell_p^n))\,,
$$
hence we conclude from Proposition \ref{golproj} that
$$
\chi_{\mon}(\mathcal{P}^m(\ell_p^n)) \leq C^m  \lambda \big(  \mathcal{P}^{m-1}(\ell_p^n ))\,,
$$
and then finally by Proposition \ref{proproj}
$$
\chi_{\mon}(\mathcal{P}^m(\ell_p^n)) \leq C^m (1 + \frac{n}{m})^{(m-1)( 1-\frac{1}{\min \{p,2\} })} \,;
$$
(here the absolute constant $C$ is of course changing  step by step). This gives Theorem \ref{mainII}.$\,\,\Box$

\vspace{2mm}

\noindent {\bf (2) Proof of Theorem \ref{mainIII}}:
 Fix some Banach space
$\ell^n_p$. From \eqref{eq2.1}
we know that
\begin{equation*}
  \frac{1}{3\sup_m \chi_{\mon}(\mathcal{P}(^m\ell_n^p))^{\frac{1}{m}}}\,
  \le K(B_{\ell_n^p})\,,
  \end{equation*}
hence by Theorem \ref{mainII} there is some absolute constant $C \geq 1$ such that  $$\chi_{\mon}(\mathcal{P}^m(\ell_p^n))^{\frac{1}{m}} \leq C \,\,\, \text{ whenever} \,\,\,m \geq n$$
and
$$\chi_{\mon}(\mathcal{P}^m(\ell_p^n))^{\frac{1}{m}}
\leq C (\frac{n}{m})^{\frac{m-1}{m}( 1-\frac{1}{\min \{p,2\} })}  \,\,\, \text{ whenever} \,\,\, n > m\,.$$
Minimizing $m n^{1/m}$ for $n > m$ then proves Theorem \ref{mainIII}.$\,\,\Box$

\vspace{3mm}

We finished with an improved definte version of \cite[Remark 1]{DeFr}  which  in the context of unconditionality quantifies the
``gap"
between symmetric and full injective tensor products of
$\ell_p^n$'s.

\begin{remark}
There is a constant $C> 0$ such that the following estimates hold for
each $1 \leq p \leq \infty$ and $n$:
\begin{itemize}
\item[(1)]
$\frac{1}{C}  \bigg( \frac{n}{\log n} \bigg) ^{\frac{1}{\max (p,2)}}
\leq
\sup\limits_m  \chi_{\mon}(\otimes^{m,s}_{\varepsilon_s} \ell^n_p)^{\frac{1}{m}}
\leq C \bigg( \frac{n}{\log n} \bigg) ^{\frac{1}{\max(p,2)}}$

\vspace{1mm}

\item[(2)]
$\frac{1}{C}\, n ^{\frac{1}{\max(p,2)}}
\, \leq
\, \sup\limits_m  \chi_{\mon}(\otimes^m_{\varepsilon}
\ell^n_p)^{\frac{1}{m}}\,
\leq \, C \,n ^{\frac{1}{\max(p,2)}}\,.$
\end{itemize}
\end{remark}

\end{document}